\documentclass{article}

\usepackage{graphicx}%
\usepackage{multirow}%
\usepackage{amsmath,amssymb,amsfonts}%
\usepackage{amsthm}%
\usepackage{mathrsfs}%
\usepackage[title]{appendix}%
\usepackage{xcolor}%
\usepackage{textcomp}%
\usepackage{manyfoot}%
\usepackage{booktabs}%
\usepackage{algorithm}%
\usepackage{algorithmicx}%
\usepackage{algpseudocode}%
\usepackage{listings}%
\usepackage{tikz-cd}
\usepackage{enumerate}

\theoremstyle{plain}
\newtheorem{theorem}{Theorem}[section]

\newtheorem{proposition}[theorem]{Proposition}
\newtheorem{corollary}[theorem]{Corollary}

\theoremstyle{definition}
\newtheorem{definition}[theorem]{Definition}
\newtheorem{notation}[theorem]{Notation}
\newtheorem{example}[theorem]{Example}

\newcommand{\calg}[0]{\mathbf{C}}
\newcommand{\clos}[0]{\mathrm{Clos}}
\newcommand{\boolc}[0]{\mathbf{Bool}}
\newcommand{\agheyc}[0]{\mathbf{augStone}}
\newcommand{\obj}[0]{\mathbf{Obj}}

\begin{document}


\title{An Adjunction Between Boolean Algebras and a Subcategory of Stone Algebras
}

\author{Inigo Incer}
\date{
Computing and Mathematical Sciences\\
California Institute of Technology\\
inigo@caltech.edu
}


\maketitle

\abstract{
We consider Stone algebras with a distinguished element $e$ satisfying the identity $e \to x = \neg \neg x$ for all elements $x$ of the algebra.
We provide an adjunction between the category of such algebras and that of Boolean algebras.
}



\section{Introduction}\label{sec:intro}

The algebra of contracts \cite{BenvenisteContractBook,Incer:EECS-2022-99} has been an object of attention in computer science and engineering due to its capability to support compositional design. Given a Boolean algebra $B$, its contract algebra $\calg(B)$ has elements $(a, g) \in B^2$ such that $a \lor g = 1_B$.
$\calg(B)$ is a Stone algebra with operations
$(a,g) \land (a',g') = (a \lor a', g \land g')$,
$(a,g) \lor (a',g') = (a \land a', g \lor g')$, and
$(a',g') \to (a,g) = ((a \land \neg a') \lor (g' \land \neg g), \neg g' \lor g)$. The top and bottom elements of this algebra are, respectively, $1 = (0_B, 1_B)$ and $0 = (1_B, 0_B)$.
This note introduces a subcategory of Stone algebras with the property that the contract functor is the left adjoint of the functor that maps these algebras to their closures.

\section{Augmented Stone Algebras}

First we recall the definition of Stone algebras. Then we introduce the concept of an augmented Stone algebra.

\begin{definition}
A Stone algebra $S$ is a Heyting algebra satisfying $\neg x \lor \clos(x) = 1_S$ for all $x \in S$, where
$\neg x = x \to 0$ and $\clos(x)$ is the closure of $x$: $\clos(x) = \neg \neg x$.
We say that $x$ is \emph{closed} if $\clos(x) = x$. We say $x$ is \emph{dense} if $\clos(x) = 1_S$.
\end{definition}

\begin{definition}
We say $(S,\land,\lor,\to, 1, 0, e)$ is an \emph{augmented Stone algebra} if
$(S,\land,\lor,\to, 1, 0)$ is a Stone algebra and
$e \in S$, called the closure element, satisfies $e \to x = \clos(x)$ for all $x \in S$.
\end{definition}

\begin{notation}
Let $\boolc$ be the category of Boolean algebras, and $\agheyc$ that of augmented Stone algebras. We define $\clos\colon \agheyc \to \boolc$ as the functor that maps an extended Stone algebra to its Boolean algebra of closed elements.
\end{notation}

\begin{example}\label{kqxgk}
Any Boolean algebra $B$ is an augmented Stone algebra with $e_B = 1_B$.
$\calg(B)$ is an augmented Stone algebra with
$e = (1_B, 1_B)$.
\end{example}

Let us prove some elementary facts about augmented Stone algebras.

\begin{proposition}
    Let $S$ be an augmented Stone algebra.
    \begin{enumerate}[\normalfont(i)]
        \item $\clos(e) = 1$.
        \item The closure element is unique.
        \item $y \to e$ is dense for all $y \in S$.
        \item If $x \in S$ is closed, $x \to y = \neg x \lor y$.
    \end{enumerate}
\end{proposition}
\begin{proof}
    \begin{enumerate}[\normalfont(i)]
        \item $1 = e \to e = \clos(e)$.
        \item Suppose $e$ and $e'$ are closure elements. Then $e \to e' = \clos(e') = 1$, whence $e \le e'$. By the same reasoning, $e' \le e$, and so $e = e'$.
        \item $\clos(y \to e) = e \to (y \to e) = y \to (e \to e) = 1$.
        \item For $a \in S$, we have
    $a \le x \to y \Leftrightarrow a \land x \le y \Leftrightarrow
    (a \land x) \lor \neg x \le y \lor \neg x \Leftrightarrow
    (a \land x) \lor \neg x \lor (\neg x \land a) \le y \lor \neg x \Leftrightarrow
    a \le y \lor \neg x
    $.\qedhere
    \end{enumerate}
\end{proof}

\begin{proposition}
$\calg$ is a functor $\calg\colon \boolc \to \agheyc$.
\end{proposition}
\begin{proof}
Given $B \in \obj(\boolc)$, we know that $\calg(B) \in \obj(\agheyc)$
from Section \ref{sec:intro} and Example \ref{kqxgk}. Given a map $f \in \hom_{\boolc}(B, B')$, we have $\calg(f)(a,g) = (fa, fg) \in \calg(B')$ as $f(a) \lor f(g) = f(a \lor g) = 1_{B'}$. Moreover, $\calg(f)(1) = 1$, $\calg(f)(0) = 0$, and $\calg(f)(e) = e$.
We also have
\begin{align*}
\calg(f)((a,g) \land (a', g')) & =
\calg(f)((a\lor a',g\land g')) =
((f(a)\lor f(a'),f(g)\land f(g'))) \\ &=
\calg(f)(a,g) \land \calg(f)(a', g'), \\
\calg(f)((a,g) \lor (a', g')) &=
\calg(f)((a\land a',g\lor g')) =
((f(a)\land f(a'),f(g)\lor f(g'))) \\ &=
\calg(f)(a,g) \lor \calg(f)(a', g'),
\intertext{and}
\calg(f)((a',g') \to (a, g)) &=
\calg(f)(((g'\to g) \to (a\land \neg a'),g\to g')) \\ &=
(((fg' \to fg) \to (fa \land \neg fa'),fg \to fg')) \\ &=
\calg(f)(a',g') \to \calg(f)(a, g).
\qedhere
\end{align*}
\end{proof}

\section{An adjunction between $\boolc$ and $\agheyc$}

Let $B \in \obj(\boolc)$ and $S \in \obj(\agheyc)$. For $(a, g) \in \calg(B)$, we let $\pi_1 (a, g) = a$ and $\pi_2 (a, g) = g$, and for $x \in B$, we let $\Delta x = (\neg x, x)$.

\begin{proposition}\label{kqhgs}
Let $f \in \hom_{\boolc} (B, \clos(S))$. The assignment
$$
\alpha_{B, S}\colon f \mapsto f \pi_2 \land (f\pi _1 \to e_S)
$$
is a set morphism $\alpha_{B, S} \colon \hom_{\boolc} (B, \clos(S)) \to \hom_{\agheyc} (\calg(B), S)$.
\end{proposition}
\begin{proof}
Let $f^* = \alpha_{B, S} f$ and $c,c' \in \calg(B)$, where $c = (a,g)$ and $c' = (a', g')$.
\begin{itemize}
\item 
$f^* (0) = f (0_B) \land (f (1_B) \to e_S) = 0_S$.
\item 
$f^* (1) = f (1_B) \land (f (0_B) \to e_S) = 1_S$.
\item 
$f^* (e) = f (1_B) \land (f (1_B) \to e_S) = e_S$.
\item 
$
\begin{aligned}[t]
f^* & \left(c \land c'\right) =
f^* (a \lor a', g \land g') =
f(g \land g') \land (f(a \lor a') \to e) \\ &=
f(g) \land f(g') \land (f(a) \to e) \land (f(a') \to e) =
f^*c \land f^*c'.
\end{aligned}
$

\item
$
\begin{aligned}[t]
f^* & \left(c \lor c'\right) =
f^* (a \land a', g \lor g') =
f(g \lor g') \land (f(a \land a') \to e) \\ &=
\left( f(g) \land (f(a) \land f(a') \to e) \right) \lor \left( f(g') \land (f(a) \land f(a') \to e) \right).
\end{aligned}
$\\
Since $a' \lor g' = 1_B$, we have $\neg a' \le g'$. Thus,
\begin{align*}
f^*  \left(c \lor c'\right) =&
\left( f(g) \land (f(a) \land f(a') \to e) \right) \lor \\ &\left( f(g') \land (f(a) \land f(a') \to e) \right) \lor 
\left( \neg f(a') \land (f(a) \land f(a') \to e) \right) \\ =&
\left( f(g) \land (f(a) \land f(a') \to e) \right) \lor \left( f(g') \land (f(a) \land f(a') \to e) \right) \\ & \lor  \neg f(a').
\end{align*}
The last equality follows from the fact that $\neg f(a') \le f(a) \land f(a') \to e$.
We conclude that
\begin{equation}\label{kqhxg}
f^*  \left(c \lor c'\right) =
\left( f(g) \land (f(a) \to e) \right) \lor \left( f(g') \land (f(a) \land f(a') \to e) \right). 
\end{equation}
By applying an analogous procedure, we obtain
\begin{equation*}
f^*  \left(c \lor c'\right) =
\left( f(g) \land (f(a) \land f(a') \to e) \right) \lor \left( f(g') \land (f(a') \to e) \right). 
\end{equation*}
Conjoining this expression with \eqref{kqhxg} yields
\begin{align*}
f^*  \left(c \lor c'\right) &=
\left( f(g) \land (f(a) \to e) \right) \lor \left( f(g') \land (f(a') \to e) \right) \\ &=
f^*c \lor f^*c'.
\end{align*}
\item
$
\begin{aligned}[t]
f^* & \left(c' \to c\right) \\ &=
f(g' \to g) \land \left( f\left( (a \land \neg a') \lor \neg (g' \to g) \right) \to e \right) \\ &=
f(g' \to g) \land \left( \neg f (a \land \neg a') \land f (g' \to g) \lor e \right) \\ &=
f(g' \to g) \land \left( f (a \land \neg a') \to e \right)
\\ &=
f(g' \to g) \land f(\neg a' \to g) \land \left( f (a \land \neg a') \to e \right)
\end{aligned}$

\noindent
The last equality follows from the fact that $\neg a' \le g'$. We have
\begin{align*}
f^* & \left(c' \to c\right) \\
 &=
(f(g') \land e \to f(g) \land (f(a) \to e)) \land (\neg f(a') \to f(g) \land (f(a) \to e)),
\intertext{which holds because $f(g' \to g) = f(g') \land (e \to f(g) \land (f(a) \to e))$. We finally obtain}
f^* & \left(c' \to c\right) 
=
(\neg f(a') \lor f(g') \land e ) \to (f(g) \land (f(a) \to e)) \\ &=
f(g') \land (f(a') \to e ) \to (f(g) \land (f(a) \to e)) \\ &=
f^* c' \to f^*c.\qedhere
\end{align*}
\end{itemize}
\end{proof}

\begin{proposition}\label{qkghskg}
Let $f^* \in \hom_{\agheyc} (\calg(B), S)$. The assignment
$$
\beta_{B, S}\colon f^* \mapsto f^* \Delta
$$
is a set morphism $\beta_{B, S} \colon \hom_{\agheyc} (\calg(B), S) \to \hom_{\boolc} (B, \clos(S))$.
Moreover,
$\alpha_{B, S}$ and $\beta_{B,S}$ are inverses.
\end{proposition}
\begin{proof}
Let $b, b' \in B$ and set $f = \beta_{B, S} f^*$.
We have $f(0_B) = f^* \Delta 0_B = f^* 0 = 0$. Now we study the Boolean operations:
$f (b \land b') = f^* (\neg (b\land b'), b\land b') = f^* b \land f^* b' = fb \land fb'$.
$f (b \lor b') = f^* (\neg (b\lor b'), b\lor b') = f^* b \lor f^* b' = fb \lor fb'$.
$f (\neg b) = f^* (\neg(b \to 0_B),b \to 0_B) = f^* ((\neg b, b)\to 0) = f b \to 0 = \neg f b$.
Regarding the second part, let $(a, g) \in \calg(B)$. We have
$
(\alpha_{B, S} \beta_{B,S} f^*) (a,g)= 
(\alpha_{B, S} f^* \Delta) (a,g)= 
(f^* \Delta \pi_2 \land (f^* \Delta \pi_1 \to e_S)) (a,g) =
f^* (\neg g, g) \land (f^* (\neg a, a) \to e_S) =
f^* \left((\neg g, g) \land ( (\neg a, a) \to e)\right) =
f^* \left((\neg g, g) \land ( (a, 1) )\right) =
f^* (a, g)
$
and
$
(\beta_{B,S} \alpha_{B, S} f) (b) =
(f\pi_2 \Delta \land (f\pi_1 \Delta \to e_S)) (b) =
f(b) \land (f(\neg b) \to e_S) =
f(b) \land (f(b) \lor e_S) =
f(b)
$.
\end{proof}

\begin{corollary}
$\calg$ is fully faithful.
\end{corollary}
\begin{proof}
Let $B$ be a Boolean algebra. 
The maps
$\Delta\colon B \to \clos (\calg (B))$ and 
$\pi_2 \colon \clos (\calg (B)) \to B$ form an isomorphism in $\boolc$. Therefore,
$$
\hom_{\boolc} (B, B) \cong
\hom_{\boolc} (B, \clos(\calg(B))) \cong 
\hom_{\agheyc} (\calg(B), \calg(B)),
$$
where the second isomorphism follows from Proposition \ref{qkghskg}.
\end{proof}

\begin{theorem}
$\calg$ is the left adjoint of $\clos$.
\end{theorem}
\begin{proof}
Given $B \in \obj(\boolc)$ and $S \in \obj(\agheyc)$, Propositions \ref{kqhgs} and \ref{qkghskg} show that
$\alpha_{B, S} \colon \hom_{\boolc} (B, \clos(S)) \to \hom_{\agheyc} (\calg(B), S)$
and
$\beta_{B, S} \colon \hom_{\agheyc} (\calg(B), S) \to \hom_{\boolc} (B, \clos(S))$
are inverses. It remains to show this isomorphism is natural in both arguments.
Let $\rho \in \hom_{\boolc}(B', B)$ and $\sigma \in \hom_{\agheyc}(S, S')$. We want to show that the following diagram commutes:
\begin{equation*}
\begin{tikzcd}
\hom_{\boolc} (B, \clos(S)) 
\arrow[r, shift left, "\alpha_{B, S}"]
\arrow[r, shift right, "\beta_{B, S}"',leftarrow]
\arrow[d, "\sigma \circ () \circ \rho"]
&
\hom_{\agheyc} (\calg(B), S)
\arrow[d, "\sigma \circ () \circ \calg(\rho)"]
\\
\hom_{\boolc} (B', \clos(S')) 
\arrow[r, shift left, "\alpha_{B', S'}"]
\arrow[r, shift right, "\beta_{B', S'}"',leftarrow]
&
\hom_{\agheyc} (\calg(B'), S')
\end{tikzcd}
\end{equation*}
Let $f \in \hom_{\boolc}(B, \clos(S))$ and $(a',g') \in \calg(B')$. We have
\begin{align*}
\left( \alpha_{B',S'} (\sigma f \rho) \right) (a',g') &=
(\sigma f \rho \pi_2 \land (\sigma f \rho \pi_1 \to e_{S'}))(a',g') \\ &=
\sigma f \rho (g') \land (\sigma f \rho (a') \to e_{S'}) \\ &=
\sigma \left( f \rho (g') \land (f \rho (a') \to e_{S}) \right) \\ &=
\sigma \left( f \pi_2 \land (f \pi_1 \to e_{S}) \right) (\rho (a'), \rho (g')) \\ &=
\sigma \alpha_{B,S}(f) \calg(\rho) (a',g').
\end{align*}
Let $f^* \in \hom_{\agheyc}(\calg(B), S)$ and $b' \in B'$. We have
\begin{align*}
\beta_{B',S'} \left( \sigma f^* \calg(\rho) \right) (b') &=
\sigma f^* \calg(\rho) \Delta (b') \\ &=
\sigma f^* (\neg \rho b',\rho b') \\ &=
\sigma f^* \Delta (\rho b') \\ &=
(\sigma \beta_{B,S} (f^*) \Delta \rho) (b').\qedhere
\end{align*}
\end{proof}

\subsection*{Funding details}

This work was supported by NSF and ASEE through an eFellows postdoctoral fellowship.






\bibliographystyle{acm}
\bibliography{references}

\end{document}